\documentclass {amsart}

\usepackage {amsthm, amssymb, amsmath}
\usepackage{hypernat}
\usepackage[T1]{fontenc}

\DeclareMathOperator {\der} {d}

\newcommand{\M}{\mathcal{M}}
\newcommand{\lf}[1]{#1^{^\leftarrow}}

\newcommand{\too}{\twoheadrightarrow} 

\newcommand{\MF}{\lf{\mathcal{M}}}
\theoremstyle {plain}
\newtheorem {theorem} {Theorem} [section]

\newtheorem {lemma} [theorem] {Lemma}
\newtheorem {corollary} [theorem] {Corollary}
\theoremstyle {definition}
\newtheorem {definition} [theorem] {Definition}

\title[Decomposable closed-graphed correspondences]{A fixed point
  theorem for closed-graphed decomposable-valued correspondences}

\thanks{We wish to thank Bill Johnson, Ben de Pagter, and Tony Wickstead for informative discussions 
on the local continuity properties of sublattices of Banach lattices.}

\author{Idione Meneghel} \address{School of Economics, The University
  of Queensland, Brisbane St Lucia QLD 4072, Australia.}

\author{Rabee Tourky} \address{School of Economics, The University of
  Queensland, Brisbane St Lucia QLD 4072, Australia.}

\date{\today}

\begin{document}

\begin{abstract}
  We prove a fixed point theorem for closed-graphed,
  decomposable-valued correspondences whose domain and range is a
  decomposable set of functions from an atomless measure space to a
  topological space. One consequence is an improvement of the fixed
  point theorem in Cellina, Colombo, and Fonda~\cite{Cel86}.
\end{abstract}

\maketitle

\section{Introduction}

Let $(S, \Sigma, \mu)$ be an atomless measure space and let $L$ be a
Banach space. Denote by $L_1(\mu,L)$ the Banach space of all
$L$-valued Bochner $\mu$-integrable functions (equivalence classes) on
$S$. A set $F\subseteq L_1(\mu,L)$ is \emph{decomposable} if for every
$f,g\in F$ and each $E\in \Sigma$ the decomposition $\chi_E g +
\chi_{S\setminus E} f$ is in $F$. A set $X\subseteq L_1(\mu, L)$ is
\emph{$\mu$-uniformly compact} if for each $\varepsilon >0$ there is
$E\in \Sigma$ with $\mu(E)<\varepsilon$ such that
\[		
  \{\chi_{S\setminus E} f \colon f\in X\}
\]
is compact when endowed with the metric $\delta(f,g) =
\|f-g\|_\infty\wedge 1$ and there is an integrable function $g\in L_1(\mu,L)$ satisfying
$\|g(s)\|\ge \|f(s)\|$ almost surely for all $f\in X$.
The following is one consequence of the theorem in
this paper. It illustrates the relation between the work here and the
literature on fixed point theory for decomposable sets.

\begin{corollary}\label{cor:intro}
  Let $F$ be a nonempty closed decomposable set in $L_1(\mu,L)$ and
  let $B\colon F\too F$ be a decomposable-valued correspondence with
  a closed graph. If there is a nonempty $\mu$-uniformly compact set
  $X\subseteq F$ satisfying $X \cap B(f) \ne\emptyset$ for each $f\in
  X$, then $B$ has a fixed point.
\end{corollary}

A related result is the fixed point theorem for continuous functions
in Cellina~\cite{Cel84}, and its generalization in
Fryszkowski~\cite{Fry04}. There $B$ is a continuous function that maps
$F$ into a norm compact, not necessarily decomposable, set
$X$. Cellina, Colombo and Fonda's~\cite{Cel86} proved the existence of
approximate continuous selections for decomposable-valued upper
hemicontinuous correspondences. One application that appears in
\cite{Cel86} of this selection result is a fixed point theorem for
decomposable-valued correspondences. However, this corollary required
that the range of the correspondence be norm compact and
decomposable. It was latter shown by Cellina and
Mariconda~\cite{Cel89} that decomposable compact sets in $L_1(\mu,L)$
contain at most one element.

The result is motivated by our study of Bayesian games \cite{MT13} and
the foundations of strategic interaction with purely subjective
uncertainty \cite{GMT13}. The proof is informed by the concepts and
techniques of Athey~\cite{Ath01}, McAdams~\cite{McA03}, and
Reny~\cite{Ren11}.

\section{The fixed point theorem}

Let $(S,\Sigma,\mu)$ be an atomless probability space and let $T$ be a nonempty
topological space. Let $L(S,T)$ be the set of all functions, not
necessarily measurable, from $S$ to $T$. Endow $L(S,T)$ with the
topology of pointwise convergence.

For any $E\in \Sigma$ and any pair of functions $f,g\in L(S,T)$ let
$g_Ef \in L(S,T)$ be define as follows:
\[
g_Ef(s)=
 	\begin{cases}
		g(s) & \text{if $s\in E$}\,,\\
		f(s) & \text{if $s\in S \setminus E$}\,,
	\end{cases}
\]
for all $s\in S$.

A subset $F\subseteq L(S,T)$ is \emph{decomposable} if for each $E\in \Sigma$ 
and $f,g\in F$ we have $g_Ef\in F$. A set-valued (possibly empty valued) mapping 
$B\colon F\too F$ is a \emph{decomposable mapping} if its domain $F$ and values $B(f)$, for
all $f\in F$, are decomposable subsets of $L(S,T)$. A decomposable mapping $B$ is
\emph{$\mu$-sequentially closed graphed} if the following hold:
\begin{enumerate}
\item If $\mu(E)=0$ and $g\in B(f)$, then $h_E g \in B(f)$ and $g\in
  B(h_E f)$ for all $h\in F$.
\item $F$ is sequentially closed in $L(S,T)$.
\item $B$ has a sequentially closed graph in $F\times F$.
\end{enumerate}
A \emph{fixed point} of $B$ is a function $f\in F$ satisfying $f\in
B(f)$.

\begin{definition}
  A subset $Y\subseteq L(S,T)$ has the \emph{$\mu$-fixed point
    property} if every decomposable $\mu$-sequentially closed-graphed
  mapping $B\colon F\too F$ that satisfies the following:
  \begin{enumerate}
  \item[(a)] $Y\cap F\ne \emptyset$.
  \item[(b)] $Y\cap B(f)\not=\emptyset$ for all $f\in Y\cap F$.
  \end{enumerate}
  has a fixed point in $Y\cap F$.
\end{definition}

A subset $X$ of $L(S, T)$ is metrizable if it is a metrizable
topological space when endowed with the topology of pointwise
convergence. Sets (of measurable functions) that are compact and
metrizable in the topology of pointwise convergence are extensively
studied in \cite{Tul73,Tul74}.

\begin{theorem}\label{thm:main}
  Each compact and metrizable set $X \subseteq L(S,T)$ is a subset of
  a sequentially compact set $Y \subseteq L(S,T)$ that has the
  \emph{$\mu$-fixed point property}.
\end{theorem}

We have the following immediate corollary.

\begin{corollary}\label{cor:sav}
  Let $B\colon F\too F$ be a decomposable $\mu$-sequentially
  closed-graphed mapping. If for a nonempty compact and metrizable $X \subseteq
  F$ we have $X\cap B(f)\ne \emptyset$ for each $f\in F$, then $B$ has
  a fixed point.
\end{corollary}

 \section{Proof of Theorem~\ref{thm:main}}

 Let $C\subseteq [0,1]$ be the Cantor ternary set. Let $L([0,1],C)$ be
 the set of all functions from $[0,1]$ to $C$. A function
 $\theta\colon L([0,1],C)\to L(S,T)$ is \emph{sequentially pointwise
   continuous} if for any sequence $f_n\in L([0,1],C)$ converging
 pointwise to $f\in L([0,1],C)$, the sequence $\theta(f_n)$ converges
 pointwise to $\theta(f)$ in $L(S,T)$. 
 Extending the notion of decomposability to
 subsets of $L([0,1],C)$, we say that $D\subseteq L([0,1],C)$ is
 \emph{decomposable} if for any pair of functions $f,g\in D$ and Borel
 set $E\subseteq [0,1]$ we have $g_Ef\in D$.

\begin{lemma}\label{lem:theta}
  There is a function $\theta\colon L([0,1],C) \to L(S,T)$ satisfying
  the following:
  \begin{enumerate}
  \item $\theta$ is sequentially pointwise continuous.
  \item $\theta$ maps the constant functions in $L([0,1],C)$ onto $X$.
  \item If $D\subseteq L(S,T)$ is decomposable, then $\theta^{-1}(D)$
    is decomposable.
  \item If $f,g\in L([0,1],C)$ differ on exactly a countable set of
    points in $[0,1]$, then $\theta(f)$ and $\theta(g)$ differ on a
    $\mu$-zero measure set.
  \end{enumerate}
\end{lemma}

\begin{proof}
  Order the members of $\Sigma$ as follows $E\le E'$ if either $E=E'$
  or $E\subseteq E'$ and $\mu(E) < \mu(E')$.  Consider a maximal chain
  $\{E_\lambda\}$ of this ordering containing $S$ and the empty set.
  Because $\mu$ is atomless $E_\lambda\mapsto \mu(E_\lambda)$ is a one
  to one onto mapping from $\{E_\lambda\}$ to $[0,1]$. So we can
  reindex the maximal chain by means of the identity $\lambda
  =\mu(E_\lambda)$.  Let $\mathbb{Q}$ be the set of rational numbers in
  $[0,1]$ and for each $s\in S$ let
  \[
  r(s)= \inf \{\lambda\in \mathbb{Q}\colon s\in E_\lambda\} = \sup
  \{\lambda\in \mathbb{Q}\colon s\notin E_\lambda\} \,.
  \]
  This is a measurable function satisfying $\mu(r^{-1}(E)) =0$ for any
  zero measure Borel subset of $[0,1]$.

  Because $X$ is compact and metrizable, the Hausdorff-Alexandroff
  Theorem says that there is a continuous function $\psi$ mapping
  $C$ onto $X$.

  Define the $\theta\colon L([0,1],C)\to L(S,T)$ as follows:
  \[
  \theta(f)(s)= \psi(f(r(s)))(s)\,,
  \]
  for all $f\in L([0,1],C)$ and $s\in S$. Note that
  \[
  \theta(g_{E} f) = \theta(g)_{r^{-1}(E)}\theta(f)
  \]
  for any $f,g\in L([0,1],C)$ and $E\subseteq [0,1]$.

  We prove that $\theta$ has the required properties:
  \begin{enumerate}
  \item If $f_n\in L([0,1],C)$ is a sequence that converges pointwise
    to $f\in L([0,1],C)$, then for any $\alpha\in [0,1]$ the sequence
    $\psi(f_n(\alpha))$ converges to $f(\alpha)$ in $X$. Thus, for all
    $s\in S$ the sequence $\psi(f_n(\alpha))(s)$ converges in
    $T$. This tells us that $\theta(f_n)$ converges pointwise to
    $\theta(f)$ in $L(S,T)$.
  \item If $f(\alpha) = c$ for all $\alpha\in [0,1]$, then
    $\theta(f)(s) = \psi(c)(s)$ for all $s$.
  \item Let $D\subseteq L(S,T)$ be a decomposable set. If $f,g\in
    \theta^{-1}(D)$ and $E$ is a Borel subset of $[0,1]$, then
    $r^{-1}(E)$ is in $\Sigma$ and $\theta(g_{E} f) =
    \theta(g)_{r^{-1}(E)}\theta(f)\in D$. Thus, $g_{E} f\in
    \theta^{-1}(D)$.
  \item If $f,g\in L([0,1],C)$ and $g= h_{E} f$ for some zero measure
    Borel set $E\subseteq [0,1]$, then $\theta(g)= \theta(h_{E} f) =
    \theta(h)_{r^{-1}(E)}\theta(f)$ and $r^{-1}(E)=0$.
  \end{enumerate}
\end{proof}

Fix a function $\theta\colon L([0,1],C)\to L(S,T)$ satisfying the
properties in Lemma~\ref{lem:theta}.  Let $\M$ be the set of monotone
functions from $[0,1]$ to $C$. This is a sequentially compact set in
the topology of pointwise convergence. Let $Y= \theta(\M)$, which is
also a sequentially compact subset of $L(S,T)$, because $\theta$ is
sequentially continuous. The set $Y$ contains $X$ because $\M$ contains
the constant functions, and $\theta$ maps the constant functions onto
$X$. We want to show that $Y$ has the $\mu$-fixed point property.

Fix a set valued mapping $B\colon F\too F$ that is decomposable,
$\mu$-sequentially upper hemicontinuous, and that satisfies the
following:
\begin{enumerate}
\item[(a)] $Y\cap F\ne \emptyset$.
\item[(b)] $Y\cap B(f)\not=\emptyset$ for all $f\in Y\cap F$.
\end{enumerate}
We need to show that $B$ has a fixed point in $Y$.

Let $\mathcal{F} = \theta^{-1}(F)$, which is a subset of $L([0, 1],
C)$, and note that it is decomposable and sequentially closed, because
of properties (3) and (1), respectively, of Lemma~\ref{lem:theta}. For
each $f\in \mathcal{F}$ let
\[
P(f) = \theta^{-1}(B(\theta(f))\,. 
\]
We record the following properties of the mapping $P\colon
\mathcal{F}\too \mathcal{F}$.

\begin{lemma}\label{lem:P}
  The following hold true:
  \begin{enumerate}
  \item $\mathcal{F}$ is sequentially closed, decomposable, and
    $\M\cap \mathcal{F}$ is non-empty.
  \item For each $f\in \mathcal{F}$, the set $P(f)$ is sequentially
    closed and decomposable.
  \item $P$ has a sequentially closed graph in
    $\mathcal{F}\times\mathcal{F}$.
  \item If $E\subseteq [0,1]$ is countable, then $g\in P(f)$ implies
    that $ h_{E} g\in P(f)$ and $g\in P(h_{E} f)$ for all $h\in
    \mathcal{F}$.
  \item For any $f\in \M\cap \mathcal{F}$, the set $P(f)\cap \M$ is
    nonempty.
  \item If $f$ is a fixed point of $P$, then $\theta(f)$ is a fixed
    point of $B$.
  \end{enumerate}
\end{lemma}

\begin{proof}
  (1) and (2) are consequences of (1) and (3) of
  Lemma~\ref{lem:theta}. (3) is a consequence of (1) of
  Lemma~\ref{lem:theta}. (4) follows from (4) of
  Lemma~\ref{lem:theta}. (5) holds because $B(f)\cap Y$ is not empty
  for any $f\in Y\cap F$. Finally, (6) holds because if $f\in P(f)$,
  then $\theta(f)\in B(\theta(f))$.
\end{proof}

So our task now is to show that $P$ has a fixed point in $\M$.

Let $\mathcal{Z}= \mathcal{F} \cap \M$, which is sequentially closed
and nonempty by (1) of Lemma~\ref{lem:P}. Define the mapping $Q\colon
\mathcal{Z}\to \mathcal{Z}$ by letting
\[
Q(f)= P(f)\cap\M\,,
\]
for all $f \in \mathcal{Z}$. This is a nonempty valued correspondence
with sequentially closed graph in $\mathcal{Z}\times\mathcal{Z}$,
because of (2) of Lemma~\ref{lem:P}. For any $f\in \M$ let $\lf{f}$ be
the right continuous version of $f$; setting $\lf{f}(1)=1$ for all
$f\in \M$. Recall that $\lf f$ differs from $f$ over a countable
subset of $[0,1]$. Also, if $\lf g = \lf f$, then $g$ differs from $f$
on a countable subset of $[0,1]$. In particular, $Q(f) = Q(g)$ and if
$f \in Q(h)$, then $g \in Q(h)$. This is, as a result of property (1)
of the definition of $\mu$-sequentially graphed mappings and (4) of
Lemma~\ref{lem:theta}.

For $\mathcal{G}\subseteq \M$, we write $\lf {\mathcal{G}}$ for the
set $\{\lf f\colon f\in \mathcal{G}\}$. For each $f\in
\lf{\mathcal{Z}}$ choose an arbitrary $g\in \mathcal{Z}$ satisfying
$\lf g = f$ and let
\[
\tilde Q(f) = \lf{Q(f)}\,.
\]
The mapping $\tilde Q\colon \lf{\mathcal{Z}}\too \lf{\mathcal{Z}}$ is
nonempty valued, because $Q$ is nonempty valued. Further, if $f$ is a
fixed point of $\tilde Q$, then any $g\in \mathcal{Z}$ satisfying $\lf
g = f$ is a fixed point of $Q$, and the required fixed point of
$P$. So we are done if we show that $\tilde Q$ has a fixed point.

Endow $\M$ with the pseudometric  
\[
	\delta(f,g) = \int_0^{1} |f(a)-g(a)| \der a\,.
\]
Notice that $(\lf{\M},\delta)$ and $(\lf{\mathcal{Z}},\delta)$ are a
compact metric spaces. Furthermore, $\tilde Q$ has a $\delta$-closed
graph in $\lf{\mathcal{Z}}\times \lf{\mathcal{Z}}$.

Order the set $\MF$ of right-continuous monotone functions by means of
the pointwise ordering whereby $f\ge g$ if $f(\alpha)\ge g(\alpha)$
for all $\alpha\in [0,1]$. The set $(\MF,\delta)$ is a
$\delta$-compact topological meet semilattice using the terminology in \cite{GHLMS80}. 

Let $\Gamma$ be the set of all nonempty closed subsets of $(\MF,\delta)$ endowed with the 
metric induced by Hausdorff distances. For any $U\subseteq \MF$ we write $\inf U$ 
for the pointwise inf of the set of functions in $U$. This is a monotone right continuous
function in $\MF$ and the infimum of the set $U$ in the lattice $\MF$. Notice that 
$\inf U=\inf \overline U$, where $\overline U$ is the closure of $U$ in $(\MF,\delta)$.
This is because if $f_n$ is a sequence in $U$ that $\delta$-converges to $f$, then 
it pointwise converges to some $g\in \M$ satisfying $\lf{g}=f$. But $g(a)\leq f(a)$ for
all $a\in [0,1]$. We will now show that the function $U\mapsto \inf U$ from $\Gamma$ to $(\MF,\delta)$ 
is continuous.

\begin{lemma}\label{lem:cont}
 If a sequence $U_n\in \Gamma$ converges to $U\in \Gamma$, then $\inf U_n$ converges to $\inf U$ in $(\MF,\delta)$. 
\end{lemma}

\begin{proof}
Let $f=\inf U$. For each $n$ let $f_n =\inf U_n$.  All of these are in $\MF$. 
Let $f^*$ be an accumulation point in $(\MF,\delta)$ of $f_n$, by moving to a 
subsequence we shall suppose that $f_n$ converges to $f^*\in \MF$. We want to 
show that $f^*=f$.
	
For each $n$ let $V_n=\cup_{m\ge n} U_m$. 
Let $g_n=\inf V_n$ for each $n$, and recall that $g_n\in \MF$. The sequence 
$g_n$ is increasing pointwise, so let $g$ be $\sup \{g_n\}$ (taking the pointwise supermum), 
which is in $\M$ but not necessarily right continuous.  
Now $U$ is in the closure of $V_n$ for each $n$. Thus,
$g_n= \inf (V_n\cup U) \leq \inf U =f$ for all $n$. In particular, $g(a)\leq f(a)$
for all $a\in [0,1]$. 

Let $\tilde f$ be the left continuous version of $f$, setting $\tilde f(0)= 0$.
 Suppose by way of contradiction that for some $a\in [0,1]$ 
we have $g(a) < d_2< d_1 < \tilde f(a)$. There is $\gamma>0$ such that
$\tilde f(a- \gamma) > d_1$. Pick $n$ large enough such that 
$\delta(h,U) < (d_1-d_2)\gamma$ for all $h\in V_n$. Pick $h\in V_n$ satisfying
$h(a) < d_2$ and $h'\in U$ satisfying $\delta(h,h') < (d_1-d_2)\gamma$.
But $d_1< \tilde f(b) \leq f(b) \leq h'(b)$ 
for all $a-\gamma\leq  b$. Thus,   $\delta(h,h') \ge (d_1-d_2)\gamma$.
This is impossible. We conclude that $g(a)\ge \tilde f(a)$
for all $a$.  Thus, $\lf{g} =f$ and $g_n$ converges to $f$. 

Now note  that $f_n\ge g_n$ for all $n$.  For each $a$,
for every $\epsilon>0$, and $n$ there is $m\ge n$ and $h\in U_m$ such that $|h(a) - g_n(a) | < \epsilon$. 
But $h(a) \ge f_m(a) \ge g_n(a)$. Thus,  $f^*=f$. 
\end{proof}

The result of Wojdys{\l}awski~\cite{Woj39} (cf. \cite{CS78}) 
tells us that when endowed with the metric induced by Hausdorff distances, 
the family of all nonempty closed subsets of a Peano continuum is an absolute
 retract. We employ this and the previous lemma to establish the next result.
  
\begin{lemma}
  If $\mathcal{G}\subseteq L([0,1],C)$ is decomposable and
  $\mathcal{G}\cap \M$ is nonempty and sequentially closed, then
  $(\lf{({\mathcal{G}\cap \M)}},\delta)$ is a compact absolute
  retract.
\end{lemma}

\begin{proof}
  The set $\lf{(\mathcal{G}\cap \M)}$ is nonempty and compact. If
  $f,g\in {\mathcal{G}}\cap \M$, then $f\wedge g$ is monotone and
  differs from $f,g$ on Borel sets.  Thus, $f\wedge g$ is in
  ${\mathcal{G}}\cap \M$. Noting that $(f\wedge g)^{\leftarrow} =
  f^\leftarrow\wedge g^\leftarrow$, we see that $(({\mathcal{G}}\cap
  \M)^{\leftarrow},\delta)$ is a sub-semilattice of $\MF$. 
  We show that it is locally connected, and thus a Peano continuum.

  First, notice that if $f,g\in {\mathcal{G}}\cap \M$ and $f\ge g$, then
  $g_{[0,\alpha)} f\in {\mathcal{G}}\cap \M$ for all $\alpha\in
  [0,1]$. Thus, if $f,g\in ({\mathcal{G}}\cap \M)^\leftarrow$, then
  $g_{[0,\alpha)} f\in ({\mathcal{G}}\cap \M)^\leftarrow$ for all
  $\alpha\in [0,1]$.

  If $U_n$ is a neighborhood base in $(({\mathcal{G}}\cap
  \M)^{\leftarrow},\delta)$ of $f\in ({\mathcal{G}}\cap
  \M)^{\leftarrow}$, then $\inf U_n$ converges to $f$ by
Lemma~\ref{lem:cont}. Thus,  
	 $V_n=\{[\inf \{U_n\},h]\colon h\in U_n\}$, where 
	 $[g,h]=\{h' \in \lf{(\mathcal{G}\cap \M)} \colon g\leq h'\leq h\}$,  is
  a neighborhood base at $f$.  Let $g = \inf U_n$ and $h\in V_n$. For
  any $\alpha\in [0,1]$ the function $g_{[0,\alpha)} h$ is in
  $V_n$. Thus, $V_n$ is path connected and  $\lf{(\mathcal{G}\cap \M)}$
  is Peano continuum. 
  
  The collection $\Gamma^*$ of nonempty closed subsets of  
  $\lf{(\mathcal{G}\cap \M)}$ is an absolute retract and the mapping
  $U\mapsto \inf U$ from $\Gamma^*$ to   $\lf{(\mathcal{G}\cap \M)}$
  is a continuous retract. This concludes the proof.
\end{proof}

The metric space $(\lf{\mathcal{Z}},\delta)$ is a
compact absolute retract, and for each $f\in \lf{Z}$ the set $\tilde
Q(f)$ is an absolute retract. Noting that $\tilde Q$ has a closed
graph in $(\lf{\mathcal{Z}},\delta) \times (\lf{\mathcal{Z}},\delta)$,
by the Eilenberg and Montgomery~fixed point theorem \cite{EM46} the
correspondence $\tilde Q$ has a fixed point. This concludes the proof of Theorem~\ref{thm:main}.

\section{Proof of  Corollary~\ref{cor:intro}} 

Let $X\subseteq F$ be the $\mu$-essentially compact subset of $L_1(\mu, L)$ and 
fix an integrable function $g\in L_1(\mu,L)$ satisfying
$\|g(s)\|\ge \|f(s)\|$ almost surely for all $f\in X$.
Let $\tilde F$ be the functions $f\in F$ satisfying $\|f(s)\|\le \|g(s)\|$
almost surely.  The restriction $\tilde B\colon \tilde F\too \tilde F$ of $B$
satisfies the conditions of Corollary~\ref{cor:intro}.

For each subset $Z$ of $L_1(\mu, L)$
we write $]Z[$ for all the measurable functions in each of equivalence classes 
of $Z$. For any $f\in ]\tilde F[$ let $P(f) = ]\tilde B(f)[$. Notice that 
$P$ is a decomposable $\mu$-sequentially closed-graphed mapping.
Corollary~\ref{cor:intro} is a consequence of Corollary~\ref{cor:sav}
and the following result. 

\begin{lemma}
If $X$ is a $\mu$-essentially compact subset of $L_1(\mu, L)$, then 
$]X[$ contains a subset that is compact and metrizable in the 
topology of pointwise convergence. 
\end{lemma}

\begin{proof}
Let $E_n$ be a decreasing sequence in $\Sigma$ satisfying 
$\mu(E_n)=\tfrac1n$ such that
\[
	\{\chi_{\Omega\setminus E_n} f\colon f\in X\}\,, 
\]
is compact for the metric $\|f-g\|_\infty\wedge 1$.  

Let $[f_n]$ be a sequence of equivalence classes in $X$ 
whose $L_1$-accumulation points comprise every function in $X$. 
Pick a version $f_n$ from each equivalent class $[f_n]$.
There exists $S'$  satisfying $\mu(S')=1$ such that for any $\ell, m,n$ and 
$s\in S'\setminus E_m$ we have
\[
	\|\chi_{S'\setminus E_m} f_\ell - \chi_{S'\setminus E_m} f_n\|_\infty \ge 
	\|f_\ell(s)- f_n(s)\|\,.
\]
Now if $f_{n'}$ is a $L_1$-converging subsequence sequence, then $f_{n'}$ 
converges pointwise on $S'$. Furthermore, if two subsequence $f_{n'}$ and $f_{n''}$
converge in $L_1$ to $[f]$, then they converge pointwise on $S'$ to the same 
version of $[f]$. This implies that there is a measurable set $S''\subseteq S'$ 
of full measure such that if $f_{n'}$ converges to $f_n$ for a fixed $n$, 
then it converges pointwise on $S''$ to $f_n$. 

So for each $[f]\in X$ we can pick a version $f$ that is zero on the complement of
$S''$ and equal to the pointwise limit on $S''$ of any $L_1$-convergent to $f$
subsequence of $f_n$. We see that the collection $\tilde X$ of these versions is 
sequentially compact,  and metrizable by means of the $L_1$-metric. 
\end{proof}

\bibliographystyle{aomplain}

\bibliography{biblio}

\end{document}